\documentclass{article}

\usepackage[
  journal=JNCG,
  lang=british,
]{ems-journal}

\makeatletter
\renewcommand*\ps@titlepage{%
  \let\@oddfoot\@empty
  \let\@evenfoot\@empty
  \let\@oddhead\@empty
  \let\@evenhead\@empty
}
\makeatother

\usepackage{tikz-cd}
\usetikzlibrary{matrix}

\theoremstyle{plain}

\newtheorem{lemma}[subsubsection]{Lemma}

\newtheorem{thmx}{Theorem}

\theoremstyle{definition}

\theoremstyle{remark}

\numberwithin{equation}{section}

\newcommand{\CC}{\mathbb{C}}

\newcommand{\ZZ}{\mathbb{Z}}
\newcommand{\PP}{\mathbb{P}}

\newcommand{\Db}{\mathrm{D}^{b}}

\newcommand{\Hom}{\mathrm{Hom}}
\newcommand{\Ext}{\mathrm{Ext}}

\begin{document}

\title{A counterexample to a global-dimension bound for weighted projective lines}

\emsauthor{1}{
	\givenname{Bochao}
	\surname{Kong}
	\mrid{}
	\orcid{}}{B.~Kong}

\emsauthor{2}{
	\givenname{Yeqin}
	\surname{Liu}
	\mrid{1591694}
	\orcid{0000-0001-8231-7682}}{Y.~Liu}

\emsauthor{3}{
	\givenname{Yu}
	\surname{Shen}
	\mrid{}
	\orcid{0000-0001-5766-1596}}{Y.~Shen}

\Emsaffil{1}{
  \department{Department of Mathematics}
  \organisation{Michigan State University}
  \rorid{05hs6h993}
  \address{619 Red Cedar Road}
  \zip{48824}
  \city{East Lansing, MI}
  \country{USA}
  \affemail{kongboc1@msu.edu}
}

\Emsaffil{2}{
  \department{Department of Mathematics}
  \organisation{University of Michigan}
  \rorid{00jmfr291}
  \address{530 Church St}
  \zip{48109}
  \city{Ann Arbor, MI}
  \country{USA}
  \affemail{yqnl@umich.edu}
}

\Emsaffil{3}{
  \department{Department of Mathematics}
  \organisation{Florida State University}
  \rorid{05g3dte14}
  \address{1017 Academic Way}
  \zip{32306-4510}
  \city{Tallahassee, FL}
  \country{USA}
  \affemail{ys26k@fsu.edu}
}

\classification[16E10, 16G20, 18G80]{14H60}

\begin{abstract}
We observe that standard derived equivalences give a counterexample to a conjecture of Kalck on global dimension for weighted projective lines.  For the root stack
$\mathcal X=\PP^1\langle \infty,0,1;2,3,3\rangle$ we exhibit a $13$-dimensional radical-square-zero algebra $A$ such that
$$
\Db(\operatorname{coh}\mathcal X)\simeq \Db(\operatorname{mod}A),
\qquad
\operatorname{gldim}A=4>3.
$$
\end{abstract}

\maketitle

\section{Introduction}\label{sec:introduction}

Global dimension is not preserved by derived equivalence. This raises the question of how large the global dimension of a finite-dimensional algebra model can become inside a fixed derived category. Schröer's \emph{Atlas of Finite-Dimensional Algebras} records the following conjecture raised by Martin Kalck \cite[Conjecture~13.43]{SchroerAtlas2023}: if $\mathcal X$ is a weighted projective line of weight type $\mathbf p=(p_1,\ldots,p_t)$ and $A$ is a finite-dimensional algebra such that
$\Db(\operatorname{coh}\mathcal X)\simeq \Db(\operatorname{mod}A)$, 
then
\begin{equation}\label{eq:proposed-bound}
\operatorname{gldim}A\leq\max_i p_i.
\end{equation}
Such a bound would impose a geometric constraint on every finite-dimensional algebra model of the same derived category. In this note we prove the following main theorem.

\begin{thmx}\label{thm:main}
There is a weighted projective line $\mathcal X$ of weight type $(2,3,3)$ and a $13$-dimensional $\CC$-algebra $A$ such that
$$
\Db(\operatorname{coh}\mathcal X)\simeq\Db(\operatorname{mod}A),
\qquad
\operatorname{gldim}A=4>3=\max\{2,3,3\}.
$$
\end{thmx}

The counterexample is obtained by combining two standard ingredients. Realize $\mathcal X$ as the root stack $\PP^1\langle\infty,0,1;2,3,3\rangle$. The domestic classification gives a hereditary model of affine type $\widetilde E_6$. For the orientation
$$
R:\quad
0\longrightarrow1\longrightarrow2\longrightarrow3\longrightarrow4,
\qquad
2\longrightarrow5\longrightarrow6,
$$
a height-shift of the vertex simples gives a tilting complex whose endomorphism algebra is
$$
A=\CC R/J^2.
$$
Thus $A$ has seven primitive idempotents and six radical generators, with all products of two radical elements zero. Its longest directed path has length four, so $\operatorname{gldim}A=4$. Hence the conjectured bound already fails by combining standard features of the domestic $\widetilde E_6$ model; the resulting counterexample has dimension only $13$.

\section*{Acknowledgments}
The second author thanks Qi Wang for many useful discussions. The counterexample was discovered with the assistance of OpenAI's GPT-5.6 Sol model; all mathematical arguments and references were independently verified by the authors.

\paragraph{Conventions.}\label{sec:notations}
All algebras are unital $\CC$-algebras and all modules are finite-dimensional right modules. We write $\Db(A)=\Db(\operatorname{mod}A)$ and use categorical composition for paths, so an arrow $i\to j$ lies in $e_jAe_i$.

\section{The weighted projective line as a root stack}\label{sec:weighted-line}

We use the algebro-geometric root-stack realization
\begin{equation}\label{eq:root-stack}
\mathcal X:=
\sqrt[2]{(\PP^1,\infty)}\times_{\PP^1}
\sqrt[3]{(\PP^1,0)}\times_{\PP^1}
\sqrt[3]{(\PP^1,1)}
=\PP^1\langle\infty,0,1;2,3,3\rangle.
\end{equation}
Thus $\mathcal X$ is the smooth proper Deligne--Mumford curve obtained by taking roots of orders $2,3,3$ along $\infty,0,1$; its coarse moduli space is $\PP^1$. See \cite{Cadman2007} for the root-stack construction.

Its orbifold Euler characteristic is
\[
\chi_{\mathrm{orb}}(\mathcal X)
=
2-\left(1-\frac12\right)
 -\left(1-\frac13\right)
 -\left(1-\frac13\right)
=\frac16>0,
\]
so $\mathcal X$ is domestic. By the standard domestic classification,
$\Db(\operatorname{coh}\mathcal X)$ is derived equivalent to a tame hereditary algebra of affine type $\widetilde E_6$ \cite{GeigleLenzing1987}. Since $\widetilde E_6$ is a tree, source and sink reflections allow us to choose any acyclic orientation \cite{Happel1988}. We may therefore choose
\begin{equation}\label{eq:rooted-quiver}
R:\quad
0\longrightarrow1\longrightarrow2\longrightarrow3\longrightarrow4,
\qquad
2\longrightarrow5\longrightarrow6,
\end{equation}
and put $H=\CC(R^{\operatorname{op}})$. Then we have
\begin{equation}\label{eq:wpl-H}
\Db(\operatorname{coh}\mathcal X)
\simeq
\Db(\operatorname{mod}H).
\end{equation}

\section{A shifted-simple tilting complex}\label{sec:shifted-simples}

The following standard shifted-simple calculation is the mechanism behind the example. It is the hereditary case of the usual Ext-algebra/Koszul-duality construction; we include the proof to fix the grading and opposite-quiver conventions.

\begin{lemma}\label{lem:shifted-simple-general}
Let $Q$ be a finite acyclic quiver without multiple arrows, and suppose there is a function
$ h:Q_0\longrightarrow\ZZ $
such that
$ h(j)=h(i)+1 $
for every arrow $i\to j$ of $Q$. Put $H_Q=\CC(Q^{\operatorname{op}})$, and let $S_i$ be the simple right $H_Q$-module at $i$. If $J\subset \CC Q$ is the arrow ideal, then
$$
T_Q=\bigoplus_{i\in Q_0}S_i[h(i)]
$$
is a tilting complex in $\Db(\operatorname{mod}H_Q)$ and
$\operatorname{End}_{\Db(H_Q)}(T_Q)\cong\CC Q/J^2$.
\end{lemma}

\begin{proof}
With the convention of Section~\ref{sec:notations}, an arrow $i\to j$ of $Q$ gives
$ \Ext^1_{H_Q}(S_i,S_j)\cong\CC. $
These are the only extension spaces between distinct vertex simples, while
$$
\Hom_{H_Q}(S_i,S_j)=0
$$
for $i\neq j$. Since $H_Q$ is hereditary, we have
$$
\Ext^r_{H_Q}(-,-)=0
\qquad
(r\geq2).
$$
Therefore
\begin{equation}\label{eq:shifted-hom}
\Hom_{\Db(H_Q)}
\bigl(S_i[h(i)],S_j[h(j)+n]\bigr)
\cong
\Ext^{h(j)+n-h(i)}_{H_Q}(S_i,S_j).
\end{equation}
If $i=j$, the right-hand side of \eqref{eq:shifted-hom} can be nonzero only for $n=0$. If $i\to j$ is an arrow, then $h(j)-h(i)=1$, so the corresponding $\Ext^1$ occurs again exactly for $n=0$. There are no other nonzero cases. Hence
$$
\Hom_{\Db(H_Q)}(T_Q,T_Q[n])=0
\qquad
(n\neq0).
$$
Moreover, $S_i=S_i[h(i)][-h(i)]$
belongs to the thick subcategory generated by $T_Q$ for every vertex $i$. The vertex simples generate $\Db(\operatorname{mod}H_Q)$, so $T_Q$ is a tilting complex.

It remains to compute its endomorphism algebra. The preceding calculation gives one primitive idempotent for each vertex and one radical generator for each arrow of $Q$. Under these identifications, every composition of two radical generators is a Yoneda product of two $\Ext^1$-classes, hence lies in an $\Ext^2$-group and is zero. So there is a surjective homomorphism
$$
\CC Q/J^2
\longrightarrow
\operatorname{End}_{\Db(H_Q)}(T_Q).
$$
Both sides have dimension $|Q_0|+|Q_1|$, so the homomorphism is an isomorphism.
\end{proof}

Apply Lemma~\ref{lem:shifted-simple-general} to $R$ with
$(h_0,\ldots,h_6)=(0,1,2,3,4,3,4)$. Then
$T=\bigoplus_{i=0}^6S_i[h_i]$ is a tilting complex over
$H=\CC(R^{\operatorname{op}})$ with endomorphism algebra
\begin{equation}\label{eq:def-A}
A:=\CC R/J^2.
\end{equation}
Thus $\dim_{\CC}A=7+6=13$, and Rickard's theorem \cite{Rickard1989}, together with \eqref{eq:wpl-H}, gives
\begin{equation}\label{eq:wpl-A}
\Db(\operatorname{coh}\mathcal X)\simeq\Db(\operatorname{mod}A).
\end{equation}

\section{Global dimension}\label{sec:global-dimension}

We now compute the global dimension of $A$.

\begin{lemma}\label{lem:rad-square-zero-gldim}
Let $Q$ be a finite acyclic quiver and let
$
B=\CC Q/J^2.
$
For a vertex $j$, the projective dimension of the simple right $B$-module $S_j$ is the maximal length of a directed path in $Q$ ending at $j$. Consequently,
$$
\operatorname{gldim}B
=
\max\{\text{length of a directed path in }Q\}.
$$
\end{lemma}

\begin{proof}
For $P_j=e_jB$ one has
$\Omega S_j=\operatorname{rad}P_j\cong\bigoplus_{i\to j}S_i$.
Iterating, $\Omega^rS_j$ is the direct sum, with multiplicities, of the simples indexed by directed paths of length $r$ ending at $j$. Hence $\operatorname{pd}S_j$ is the maximal length of such a path. Taking the maximum over the vertex simples gives the formula for $\operatorname{gldim}B$.
\end{proof}

To prove Theorem~\ref{thm:main}, note that we have \eqref{eq:wpl-A} and $\dim_{\CC}A=13$.
For $R$, the paths
$0\to1\to2\to3\to4$ and $0\to1\to2\to5\to6$ have length four, and no longer directed path exists. Hence Lemma~\ref{lem:rad-square-zero-gldim} gives
\begin{equation}\label{eq:gldim-A}
\operatorname{gldim}A=4>3=\max\{2,3,3\}.
\end{equation}

\bibliographystyle{emss}
\bibliography{ref}

@incollection{GeigleLenzing1987,
  author    = {Geigle, Werner and Lenzing, Helmut},
  title     = {A class of weighted projective curves arising in representation theory of finite-dimensional algebras},
  booktitle = {Singularities, Representation of Algebras, and Vector Bundles},
  series    = {Lecture Notes in Mathematics},
  volume    = {1273},
  publisher = {Springer},
  year      = {1987},
  pages     = {265--297}
}

@book{Happel1988,
  author    = {Happel, Dieter},
  title     = {Triangulated Categories in the Representation Theory of Finite-Dimensional Algebras},
  series    = {London Mathematical Society Lecture Note Series},
  volume    = {119},
  publisher = {Cambridge University Press},
  year      = {1988}
}

@article{Rickard1989,
  author  = {Rickard, Jeremy},
  title   = {Morita theory for derived categories},
  journal = {J. London Math. Soc. (2)},
  volume  = {39},
  year    = {1989},
  pages   = {436--456}
}

@unpublished{SchroerAtlas2023,
  author = {Schröer, Jan},
  title  = {Atlas of Finite-Dimensional Algebras},
  year   = {2023},
  note   = {Version of 29 December 2023, available from the author's webpage}
}

@article{Cadman2007,
  author  = {Cadman, Charles},
  title   = {Using stacks to impose tangency conditions on curves},
  journal = {Amer. J. Math.},
  volume  = {129},
  number  = {2},
  year    = {2007},
  pages   = {405--427}
}

\end{document}